\documentclass[a4paper]{amsart}
\usepackage{a4wide}
\usepackage[utf8]{inputenc}
\usepackage[english]{babel}

\usepackage{amsmath,amsfonts,amsthm,amssymb}
\usepackage{verbatim}
\usepackage{url}

\usepackage{tikz-cd}
\usepackage{textcomp}
\usepackage{xcolor}

\DeclareMathOperator{\Q}{\mathbb{Q}}
\DeclareMathOperator{\N}{\mathbb{N}}

\DeclareMathOperator{\C}{\mathbb{C}}
\DeclareMathOperator{\Sym}{\operatorname{Sym}}
\DeclareMathOperator{\Ps}{\mathbb{P}}
\DeclareMathOperator{\Proj}{\operatorname{Proj}}
\DeclareMathOperator{\Gal}{\operatorname{Gal}}

\renewcommand{\H}{\operatorname{H}}
\newcommand{\bd}{_\textup{bd}}

\numberwithin{equation}{section}

\newtheorem{Theorem}{Theorem}[section]

\newtheorem{conj}[Theorem]{Conjecture}
\newtheorem{Lemma}[Theorem]{Lemma}
\newtheorem{Proposition}[Theorem]{Proposition}
\newtheorem{Definition}[Theorem]{Definition}
\newtheorem{Remark}[Theorem]{Remark}

\makeatletter
\newtheorem*{rep@theorem}{\rep@title}
\newcommand{\newreptheorem}[2]{%
	\newenvironment{rep#1}[1]{%
		\def\rep@title{#2 \ref{##1}}%
		\begin{rep@theorem}}%
		{\end{rep@theorem}}}
\makeatother
\newreptheorem{theorem}{Theorem}

\AtBeginDocument{}

\begin{document}
	
\title[Batyrev--Tschinkel for a cubic surface and its symmetric square]{The Batyrev--Tschinkel conjecture for a non-normal cubic surface and its symmetric square}

\author{Nils Gubela}

\address{University of Vienna, Oskar-Morgenstern-Platz 1, 1090 Vienna, Austria }
\email{nils.gubela@univie.ac.at}

\author{Julian Lyczak}

\address{IST Austria\\
	Am Campus 1\\
	3400 Klosterneuburg\\
	Austria}
\email{jlyczak@ist.ac.at}


\begin{abstract}
	We complete the study of points of bounded height on irreducible non-normal cubic surfaces by doing the point count on the cubic surface $W$ given by $t_0^2 t_2 = t_1^2 t_3$ over any number field. We show that the order of growth agrees with a conjecture by Batyrev and Manin and that the constant reflects the geometry of the variety as predicted by a conjecture of Batyrev and Tschinkel. We then provide the point count for its symmetric square $\Sym^2 W$. Although we can explain the main term of the counting function, the Batyrev--Manin conjecture is only satisfied after removing a thin set. Finally we interpret the main term of the count on $\Sym^2(\Ps^2 \times \Ps^1)$ done by Le Rudulier using these conjecture.
\end{abstract}

\maketitle

\date{\today}

\maketitle

\thispagestyle{empty}
\setcounter{tocdepth}{1}

\section{Introduction}

In this paper we are counting points of bounded height on certain varieties. Much progress has been made in recent years, but even for cubic surfaces there are no general results. Let us consider the relevant conjectures.

\subsection{The Batyrev--Manin conjecture}
Let $W\subseteq \Ps^3$ be a cubic surface over a global field $K$. We are interested in the order of growth of counting function $N_{W,K}(B)$, i.e. the number of points in $W(K)$ whose height is bounded by $B$. Such problems are usually heavily dependent on the geometry of the variety $W$. For example, if $W$ is smooth it contains precisely $27$ lines $\Lambda_i$ and the rational points on $W$ tend to lie on these lines, in the sense that the limit $\sum N_{\Lambda_i,K}(B)/N_{W,K}(B)$ tends to $1$. This reduces the problem to well-known results on counting points on curves. A more interesting conjecture is one by Batyrev and Manin \cite{BaMan}.

\begin{conj}[The Batyrev--Manin conjecture {\cite[Conjecture C\textquotesingle]{BaMan}}]\label{conj:Manin}
Consider a smooth projective variety $V$ over a number field $K$ such that the canonical class $K_V$ is non-effective. Endow $V(K)$ with a height function coming from an adelic metrization on a big line bundle $L$. Then the Batyrev--Manin conjecture predicts that there exists an open subvariety $U \subseteq V$ for which we have
\[
N_{U,K}(B) \sim c(V) B^a \log^{b-1} B
\]
for explicit constants $a$ and $b$ which depend on the smooth projective variety $V$ and the adelically metrized line bundle $L$.

For the case $L=\omega_V^\vee$ the constant $c(V)$ has been predicted by Peyre \cite{PeyreConstant}.
\end{conj}

Note that one might need to reduce to counting the points in the ample locus of $L$, to assure that there are only finitely many points of bounded height.

The conjecture predicts a specific growth of the number of points of bounded height, but this order might only be obtained after removing an \textit{accumulating subvariety} which we will call a \textit{thin set of type I}. Since its first appearance this conjecture has been refined to include the existence of other types of accumulating subsets of $V(K)$, namely the \textit{thin sets of type II}. Even then the conjecture above is not stated in its largest generality, but it will suffice for our purposes. In particular it is expected to hold for smooth cubic surfaces; it is predicted that for $W_0$, the complement of the lines on a smooth cubic surface $W$, we have
\[
N_{W_0,K}(B) \sim c(W) B^a \log^{b-1} B
\]
for explicit constants which depend on the smooth compactification $W$ of $W_0$. Although lower bounds are known for $N_{W_0,K}(B)$ under some conditions, see for example \cite{FrSo}, the Batyrev--Manin conjecture has not been established for a single example of a smooth cubic surface.

The Batyrev--Manin conjecture also applies to singular cubic surfaces by considering their minimal desingularisation. Unlike for the smooth case it has been confirmed in several instances after excluding the singular points and the finitely many lines on the surface. For an overview of the current progress we refer to \cite{LeBou}. A noteworthy example is the singular cubic surface with isolated singularities given by $x_0^3=x_1x_2x_3$. Many more authors have previously worked on this toric surface over $\Q$, while de la Bretèche \cite{Regis} and de la Bretèche and Swinnerton-Dyer \cite{RegisSD} have obtained the best error terms.  For a more extended history on this problem consult the introduction of \cite{Frei}. An important proof for the conjecture for this surface was given by Batyrev and Tschinkel \cite{Toric}; they proved the Batyrev--Manin--Peyre conjecture for the general class of toric varieties. A new proof for this one surface was given by Frei \cite{Frei}. He used universal torsors to count points over arbitrary number fields. This is based on previous work by Derenthal and Janda \cite{DeJa} who first used this technique over all imaginary number fields with class number one. Salberger first used the universal torsor method to give a new proof of Manin’s conjecture for split toric varieties over $\Q$ \cite{PS}. Universal torsors were first introduced by Colliot-Thélène and Sansuc in \cite{CTS1} and \cite{CTS2} in the context of the Hasse-principle.
We will use Frei's approach to establish the point count in Theorem~\ref{thm:introductionV} on a cubic surface with non-isolated singularities.

For such non-normal surfaces $W$ the geometry is different from before; there are now infinitely many lines which all contribute to the main term as before. This shows that just counting points on the separate lines need not give the expected growth. Furthermore each point lies on at least one line and removing the lines does not leave anything to count. We will however establish the conjecture for a non-normal cubic surface and show that the expected order of growth is obtained by counting points of bounded height on any non-empty open subvariety. That does however not address the expected constant $c(W)$. The conjecture by Peyre does no longer apply since any singularisation $X \to W$ is no longer crepant, i.e.~$\mathcal O_W(1)$ does not pull back to the anticanonical line bundle on $X$. In such situations a conjecture by Batyrev and Tschinkel \cite{BaTsch} predicts the constant.

\begin{conj}[{The Batyrev--Tschinkel conjecture \cite{BaTsch}}]\label{conj:BT}
Let $V$ be a smooth quasi-projective $K$-variety with an associated height function $H \colon V(K) \to \mathbb R$ coming from an adelic metrization on an ample line bundle $L$. Consider a natural projective completion $V \subseteq W$, see \cite[Definition~2.1.6]{BaTsch}, and fix a desingularisation $\rho \colon X \to W$ and consider the line bundle $M := \rho^*\mathcal O_W(1)$ on $X$. Let $V_\alpha$ be the minimal closed subvarieties of $V$ for which
\[
\theta_\alpha := \lim_{B \to \infty} \frac{N_{V_\alpha,K}(B)}{N_{V,K}(B)} > 0.
\]
Suppose that
\begin{enumerate}
\item there is no closed subvariety $V' \subsetneq V$ which contains all $V_\alpha$, and
\item $M$ restricted to the closure $X_\alpha$ of $V_\alpha$ in $V \subseteq X$ is close the anticanonical bundle in the sense of \cite[Definition 2.3.4]{BaTsch}.
\end{enumerate} Then
\[
N_{V,K}(B) \sim \sum_\alpha c(X_\alpha) B^a \log^{b-1} B,
\]
where the constants $a$ and $b$ are those in the Batyrev--Manin conjecture and the constants $c$ are defined similar to those in the conjecture by Peyre.
\end{conj}

We will refer to this conjecture as \textit{the Batyrev--Tschinkel conjecture}. The complete statement together with all definitions can be found in Sections~2 and 3 in \cite{BaTsch}. This conjecture generalises the conjecture by Peyre for smooth cubic surfaces above, since one can show that for the complement of the lines $W_0$ there is only one subvariety with $\theta >0$, namely $W_0$ itself.

It does not directly apply to non-normal varieties $W$. This can be resolved by considering the regular locus $V \subseteq W$; since the complement is still lower dimensional this does not present any problems.

This conjecture was proven to be correct for the non-normal Cayley ruled surface $t_0t_1t_2=t_0^2t_3+t_1^3$ by de la Bretèche, Browning and Salberger \cite{BBS}. In this case the infinitely many lines all contribute to the main term precisely in the way as predicted by the conjecture.

\subsection{The remaining non-normal cubic surface}
In the present paper will give a count for the number of rational points of bounded height for the other non-normal cubic surface $W$ given by $t_0^2 t_2 = t_1^2 t_3$. The techniques by Frei described above will allow us to do so over arbitrary number fields. Establishing Manin's conjecture for this surface follows from work by Batyrev and Tschinkel \cite[Section~4.5]{BaTsch} since $W$ is toric.  However, we will also need explicit error terms and their approach does yield those.

We will first need to make our height function explicit. For a point $\mathbf x = (x_0\colon \dots \colon x_{n}) \in \Ps^n(K)$, we will use the \textit{usual height}
$$
H_K(\mathbf x) = \prod_{\nu \in M_K} \max_{i} \left\{ \left| x_i \right|_{\nu} \right\}^{[K_\nu : \Q_p]},
$$
where $M_K$ denotes the set of places of $K$ and $\nu$ extends the place $p$ of $\Q$. With this height, we define $Z_K(\Ps^1, s) = \sum_{\mathbf x \in \Ps^1(K)} H_K(\mathbf x)^{-s}$ to be the \textit{height zeta function} of the projective line and $c_{\Ps^1,K}$ to be the constant $N_{\Ps^1,K}(B) \sim c_{\Ps^1,K} B^2$ obtained by Schanuel \cite{Schanuel}.

We can now state our result for counting points on our non-normal cubic $W$. To compare the result with the conjecture by Batyrev and Tschinkel we will restrict to counting points on the regular locus $V$ of $W$.

\begin{Theorem}\label{thm:introductionV}
	Let $K$ be a number field and let $W$ be the surface in $\Ps^3$ given by $t_0^2 t_2 = t_1^2 t_3$.
\begin{enumerate}
\item[(a)] The regular locus $V$ of $W$ is the complement of the \textup{singular line} given by $t_0 = t_1 = 0$.
\item[(b)] The morphism $V \to \Ps^1$ given by $(t_0 \colon t_1 \colon t_2 \colon t_3) \mapsto (t_0 \colon t_1)$ is a fibration of the surface into lines. The fibre over a point $\mathbf y \in \Ps^1(K)$ will be denoted by $V_{\mathbf y}$.
\item[(c)] The surface contains one more line given by $t_2=t_3=0$ and is called the \textup{base line} $V_0$.
\item[(d)] 	We have
	$$
	N_{V,K}(B) := \{\mathbf x\in V(K)\ \colon\ H_K(\mathbf x) \leq B\} = c_{V,K} B^2+ O\left(B^{2-1/d}\right).
	$$
	The constant $c_{V,K}$ is given by
	$$
	c_{V,K} = c_{\Ps^1,K}\left(Z_K(\Ps^1,3)+1\right).
	$$
\end{enumerate}
\end{Theorem}

The main novelty in this theorem is the explicit error term. Furthermore, we also managed to make the dependency of the implicit constant on the number field explicit.

Theorem~\ref{thm:introductionV} reflects the geometry of $V$ in the following sense. Each fibre $V_{\mathbf y}$ of $V \to \Ps^1$ is an affine line with a twisted form of the usual height function. We will exhibit a smooth proper desingularisation $V \hookrightarrow X$ such that the closure of $V_{\mathbf y}$ is smooth and hence a projective line. This accounts for the $c_{\Ps^1,K}Z_K(\Ps^1,3)$. The height function on the remaining two line $V_0$ is the usual height function on $\Ps^1$ and explains the remaining $c_{\Ps^1,K}$.

Let us compare this with another expectation from \cite{BaTsch}. In the notation of Conjecture~\ref{conj:BT}, there is a natural fibration $X \to Y$ and Batyrev and Tschinkel predicted that each \textit{$V_\alpha$} will be contained in a fibre $X_{\mathbf y}$. Our example shows this to be false. So, although all the lines on $V$ form what Batyrev and Tschinkel call an \textit{asymptotic arithmetic fibration}, they do not form an actual fibration.

\subsection{Examples coming from symmetric squares}

We also give a treatment of $\Sym^2 V$ over the rationals. Its rational points are Galois invariant pairs of $\bar\Q$-points on $V$, and a height on $V$ induces a natural height on $\Sym^2 V$. Recall that counting points on $V$ can be explained by thinking of $V$ as a collection of lines. Any such line $V_{\mathbf y}$, for $\mathbf y \in \Ps^1(\Q) \cup \{0\}$, induces a closed subscheme $V_{\mathbf y,\mathbf y} \subseteq \Sym^2 V$ whose natural compactification $X_{\mathbf y,\mathbf y}$ is isomorphic to $\Ps^2$. However, the induced height on such a projective plane $\Sym^2\Ps^1$ is not again the usual height on $\Ps^2$. We will however be able to compare those results to the point count
\[
N_{\Sym^2 \Ps^1,K}(B) \sim c_{\Sym^2 \Ps^1,K} B^3
\]
when the height function does come from the usual one on $\Ps^1$. This result was established by Schmidt \cite{Schmidt} and interpreted in terms of the conjectures by Manin--Batyrev and Peyre by Le Rudulier \cite{RudulierThesis}.

The main contribution for the count for $\Sym^2 V$ comes again from the one-dimensional family of closed subschemes $V_{\mathbf y,\mathbf y}$ and the closed subschemes $V_{0,0}$ which meets every fibre of this fibration.

\begin{Theorem}
	Consider the symmetric square $\Sym^2 V$ for the surface $V$ in Theorem~\ref{thm:introductionV} defined over the rationals numbers. We have
	$$
	N_{\Sym^2 V, \Q}(B) = cB^3 + O\left( B^2 \log B\right),
	$$
	where the constant $c$ is given by
	$$
	c_{\Sym^2 \Ps^1,\Q}\left( Z_{\Q}\left(\Ps^1,9\right)+1\right).
	$$
\end{Theorem}

This agrees with Conjecture~\ref{conj:BT}, since the constant can be explained by a family of closed subschemes $V_{\mathbf y, \mathbf y} := \Sym^2 V_{\mathbf y}$ and one other $V_{0,0} := \Sym^2 V_0$. However to apply the conjecture one has to reduce to the closed subschemes $V_{0,0}$ and a certain subscheme $\mathbb V$ containing all $V_{\mathbf y, \mathbf y}$, since these do satisfy the first condition of Conjecture~\ref{conj:BT}.

One can compare this to a count done by Le Rudulier \cite{RudulierThesis}
\[
N_{\Sym^2(\Ps^1 \times \Ps^1),\Q}(B) = 2c_{\Sym^2 \Ps^1,\Q} Z_{\Q}\left(\Ps^1,6\right)B^3 + O\left( B^2 \log^3 B\right)
\]
as $B$ goes to infinity. The variety $\Ps^1\times \Ps^1$ has two fibrations into projective lines. Using this we have two families of projective planes in the form of $\Sym^2\Ps^1$ on $\Sym^2(\Ps^1 \times \Ps^1)$. Hence, there are even two positive dimensional families of targets for this counting problem and one can check that the result agrees with the expected behaviour predicted by Batyrev and Tschinkel.

This is particularly interesting since counting points on $\Sym^2(\Ps^1 \times \Ps^1)$ away from a closed subscheme gives $cB^2 \log^2 B$, as predicted by Manin's conjecture. We can however also explain the main term of this count using the Batyrev--Tschinkel conjecture, since their work \cite{BaTsch} naturally deals with such a thin sets of type~I. However, thin sets of type~II will also have to be incorporated since they are already needed in some cases to make sure the conjectured constant by Peyre is correct.

\subsection{Acknowledgements}

The authors are grateful for the continuous support of Tim Browning who pointed out the problem and supervised the Masters thesis of the first named author which contains some of the results of the present article. The second author received funding from the European Union's Horizon 2020 research and innovation programme under the Marie Sk\l odowska-Curie grant agreement No.~754411.

\section{The geometry of the non-normal cubic surface}

Let's consider the geometry of the non-normal surface.

\begin{Lemma}
Consider the scheme $W \subseteq \Ps^3$ given by $t_0^2t_2=t_1^2t_3$. The surface $W$ is non-normal and its singular locus is the \textit{singular line} $W_{\infty}$ given by $t_0=t_1=0$.
\end{Lemma}

We will write $V=W\backslash W_\infty$ for the regular locus of $W$. To explain the count of rational points of bounded heights on $V$ we will need a specific normal compactification of $V$, i.e.~a desingularisation of $W$.

\begin{Theorem}\label{thm:geometryW}
Let $x_0, x_1, x_2$, and $y_0, y_1$ be homogeneous coordinates on $\Ps^2$ and $\Ps^1$ and define $X$ to be the bilinear subscheme $X \subseteq \Ps^2 \times \Ps^1$ given by $x_0y_1=x_1y_0$.

The morphism $\rho \colon X \to W$ given by
\[
(x_0 \colon x_1 \colon x_2 \mid y_0 \colon y_1) \mapsto (x_2y_0 \colon x_2y_1 \colon x_1y_1 \colon x_0y_0)
\]
is the minimal desingularisation and the normalisation of $W$. It restricts to an isomorphism $X \backslash \{x_2=0\} \to V$.
\end{Theorem}

\begin{proof}
The equality
\[
(x_2y_0 \colon x_2y_1 \colon x_1y_1 \colon x_0y_0) = (x_2x_0 \colon x_2x_1 \colon x_1^2 \colon x_0^2)
\] shows that $\rho$ defines an actually morphism on $X$. Its birational inverse is given by
\[
(t_0 \colon t_1 \colon t_2 \colon t_3) \mapsto (t_0 \colon t_1 \colon \frac{t_1^2}{t_2}t_1t_0 \colon t_1)
\]
which is defined on $V$ since $\frac{t_1^2}{t_2} = \frac{t_0^2}{t_3}$. This proves that $\rho$ is a proper birational map and hence a desingularisation of $W$. Since the fibres of $\rho$ are all $0$-dimensional we conclude that it is even the minimal desingularisation \cite[Corollary~27.3]{LipmanRS}.

We will now show that $\rho$ is the normalisation of $W$. The morphism $\rho$ is quasi-finite and hence finite, since it is proper. From this we conclude that $\rho$ is the normalisation of $W$ by \cite[0AB1]{stacks-project}.
\end{proof}

This shows that $V$ contains many lines, which should contribute to the point count. Let us make this precise.

\begin{Definition}\label{dfn:fibrationmap}
Let $\varphi \colon X \to \Ps^1$ be the morphism given by the projection $\Ps^2 \times \Ps^1 \to \Ps^1$. The fibre of the composition $V \subseteq X \to \Ps^1$ over a point $\mathbf y \in \Ps^1(\Q)$ will be denoted by $V_{\mathbf y}$.

The line on $V$ given by $t_2=t_3=0$ will be denoted by $V_0$ and be called the \textup{base line}.
\end{Definition}

Every fibre of $\varphi$ is a projective line and hence every $V_{\mathbf y}$ is a linearly embedded affine line. This fits with Conjecture~2.4.3 in \cite{BaTsch}, which state that in this case there should be a natural fibration $X \to Y$. This is the motivation for our notation $\mathbf y=(y_1 \colon y_2)$ for a point on the base $\Ps^1$.

\subsection{The height function}
Next we will need to understand the height function on $V(K)$ in geometric terms for any number field $K$. Let us first fix for every place $\nu$ of $K$ the valuation $|x|_\nu = |N_{K_\nu/\Q_p}(x)|_p^{-[K_\nu \colon \Q_p]}$ where $p$ is the rational prime lying under $\nu$. We will use this to define an adelic metrization of the line bundle $M=\mathcal O_W(1)$ on $W$. Let $s_0, s_1, s_2, s_3$ be a basis for the global sections of $M$ on $W$ such that $s_i$ corresponds to the variable $t_i$. For a place $\nu$, a point $\mathbf x_\nu \in W(K_\nu)$ and a section $s \in \Gamma(M)$ non-vanishing at $\mathbf x_\nu$ we define
\[
\|s(\mathbf x_\nu)\|_\nu := \inf |(s/s_i)(\mathbf x_\nu)|_\nu^{[K_\nu \colon \Q_p]},
\]
where the infimum is taking over all sections $s_i$ for which $s_i(\mathbf x_\nu)$ is non-zero. This defines an adelic metric on $M$ and we get an associated height function
\[
W(K) \to \mathbb R, \quad \mathbf x \mapsto H_K(\mathbf x) := \prod_{\nu \in M_K} \|s(\mathbf x)\|^{-1}_\nu
\]
defined using a section $s \in \Gamma(M)$ non-vanishing at $\mathbf x$. This is independent of the choice of $s$ and agrees with the height function
\begin{equation}\label{eq:height}
\mathbf x \mapsto H_K(\mathbf x) = \prod_{\nu \in M_K} \max_{i} \left\{ \left| x_i\right|_{\nu} \right\}^{[K_\nu \colon \Q_p]},
\end{equation}
mentioned in the introduction. We will be interested in
\[
N_{V,K}(B) := \#\{ \mathbf x\in V(K)\; |\; H_K(\mathbf x) \leq B\}
\]
as $B$ goes to infinity.

The asymptotic behaviour for such counting problem is predicted by Conjecture~\ref{conj:Manin}. Namely, one expects that $N_{V,K}(B) \sim c B^a \log^{b-1} B$. Here the conjectured constants $a$, $b$ and $c$ depend on the geometry of $V$ together with the line bundle $L = \left.M\right|_V$ endowed with a Hermitian metric on $V_{\C}$, which we will denote by $\|.\|_h$. So let us write $\mathcal L=(L,\|.\|_h)$ for the restriction to $V$ of this ample metrized invertible sheaf $(M,\|.\|_h)$.

\subsection{The $\mathcal L$-primitive closure}
We now follow Batyrev and Tschinkel to check that $X$ is actually the natural compactification of $V$ given $\mathcal L$. To that end we introduce for each $k \geq 0$ the metrized line bundle $\mathcal L^{\otimes k}=(L^{\otimes k},\|.\|_{h,k})$, here $\|.\|_{h,k}$ is a Hermitian metric on $L^{\otimes k}$ such that for a section $s\in \Gamma(L)$ and a point $\mathbf x \in V(\C)$ we have
\[
\|s^k(\mathbf x)\|_{k,h} = \|s(\mathbf x)\|^k_h. 
\]

\begin{Definition}
Let $\H^0\bd(V,\mathcal L^{\otimes k})$ be the set of sections $s \in \H^0(V,L^{\otimes k})$ for which $\|s\|_{h,k}$ is bounded on $V(\C)$. Also set
\[
A(V,\mathcal L)=\bigoplus_{k\geq 0} \H^0\bd(V,\mathcal L^{\otimes k}).
\]
\end{Definition}

The conjecture of Batyrev and Tschinkel is formulated in terms of $\Proj A(V,\mathcal L)$.

\begin{Lemma}
There is a natural identification between $X$ and $\Proj A(V,\mathcal L)$.
\end{Lemma}

\begin{proof}
We follow the proof of Proposition 2.3 in \cite{BBS}. Details can be found there. By compactness of $X(\C)$ we have that the natural morphism
\[
\H^0(X,\mathcal O^{\otimes k}_X(1,1)) = \H^0(X, \varphi^*M^{\otimes k})\to \H^0(V,L^{\otimes k}) 
\]
factors through $\H^0\bd(V,\mathcal L^{\otimes k})$. Since $X$ is normal we can extend any bounded section on $V$ to a global analytic section on $X(\C)$, but such a section is also algebraic.

Since $\mathcal O_X(1,1)$ is ample we see from \cite[0C6J]{stacks-project} that the natural morphism
\[
X \to \Proj\left( \bigoplus_k \H^0\left(X,\mathcal O^{\otimes k}_X(1,1)\right)\right) \cong \Proj A(V,\mathcal L)
\]
is an isomorphism.
\end{proof}

\section{The conjectured behaviour of $N_{V,K}(B)$}

Since $V$ has infinitely many closed subschemes $V_{\mathbf y}$ which satisfy
\[
0 < \lim_{B \to \infty} \frac{N_{V_{\mathbf y},K}(B)}{N_{V,K}(B)} < 1
\]
there is no locally closed subscheme $V' \subsetneq V$ such that
\[
\lim_{B \to \infty} \frac{N_{V',K}(B)}{N_{V,K}(B)} = 1.
\]
In the language of \cite{BaTsch} we see that $V$ is \textit{weakly $\mathcal L$-saturated}; there is no lower dimensional subscheme $V' \subseteq V$ which explains the main term of the counting function. On the other hand, $V$ is not \textit{strongly $\mathcal L$-saturated}; there are infinitely many lower dimensional subschemes which do contribute to the main term.

\begin{Proposition}
The $\mathcal L$-targets for $V$ are the lines $V_{\mathbf y}$ together with $V_0$. 
\end{Proposition}

\begin{proof}
It is well-known that a rational curve on $V$ contibutes more to the point count than any curve of positive genus. For any rational curve of degree $d$ the contribution to the point count is $B^{\frac2d}$. Hence only the linearly embedded rational curves on $V \subseteq \Ps^3$ contribute to the main term $B^2$.  One can explicitly compute the Fano variety of lines of $V$ and show that it consists of a one dimensional component and a point, corresponding respectively to the $V_{\mathbf y}$ and $V_0$. 
\end{proof}

Note that $V_0$ and the $V_{\mathbf y}$ are strongly $\mathcal L$-saturated. So we can compute $c(V)$ as $c(V_0)+\sum_{\mathbf y \in \Ps^1(K)} c(V_{\mathbf y})$. We will compare each $c(V_{\mathbf y})$ with the $c(\Ps^1_K)$ where we have endowed $\Ps^1$ with the adelic line bundle $(\mathcal O(1),\|.\|_{\Ps^1,\nu})$ given by the minimal norm on all places.
\begin{Proposition}\label{prop:VinBTterms}
We have
\[
a(V_{\mathbf y})=a(V_0)=a(\Ps^1_K)=2,
\]
\[
b(V_{\mathbf y})=b(V_0)=b(\Ps^1_K)=1
\]
and
\[
c(V_0)=c(\Ps^1_K) \text{ and } c(V_{\mathbf y})=\frac{c(\Ps^1_K)}{H_{\Ps^1}(\mathbf y)^3}.
\]
\end{Proposition}

\begin{proof}
The expected constant is defined \cite[Section~3.4]{BaTsch} as
\begin{equation}\label{eq:expectedconstant}
c(V_{\mathbf y}) := \frac{\gamma(V_{\mathbf y})}{a(V_{\mathbf y})\left(b(V_{\mathbf y})-1 \right)!} \delta(V_{\mathbf y}) \tau(V_{\mathbf y})
\end{equation}
All of these constants except $\tau$ only depend on a chosen desingularisation of the $\mathcal L$-primitive closure. Let us consider those constants first.

The polarized schemes $(V_0,\left.\mathcal L\right|_{V_0},\|.\|_h)$ and $(\Ps^1, \mathcal O(1),\|.\|_{\Ps^1,h})$ are isomorphic, so we will only need to consider the affine lines $V_{\mathbf y}$.  Also the $\mathcal L$-primitive closure $X_{\mathbf y}$ of $V_{\mathbf y}$ is isomorphic to~$\Ps^1$. Under this identification the line bundle $\left.L\right|_{V_{\mathbf y}}$ pulls back to $\mathcal O(1)$. Hence $a(V_{\mathbf y})=a(\Ps^1)=2$ and $b(V_{\mathbf y})=b(\Ps^1)=1$ for all $\mathbf y$ and those values are well-known and easily deduced from Definitions~2.2.4 and~2.3.11 in \cite{BaTsch}. Also, $\gamma(V_{\mathbf y})$ and $\delta(V_{\mathbf y})$ only depend on a chosen desingularisation and we get
\[
\gamma(V_{\mathbf y})=\gamma(V_0)=\gamma(\Ps^1)=1 \quad \text{ and } \quad \delta(V_{\mathbf y})=\delta(V_0)=\delta(\Ps^1)=1.
\]

The last constant $\tau(V_{\mathbf y})$ does depend on $\mathbf y$, since it is defined in terms of the Tamagawa measures $\omega_{V_{\mathbf y},\nu}$ on $V_{\mathbf y}(K_\nu)$ which agree with the ones introduced in \cite[p. 112]{PeyreConstant}. We will compute the volumes with respect to the induced measures using the isomorphism $\rho_{\mathbf y} \colon \Ps^1 \to X_{\mathbf y}$. Here we pick $\rho_{\mathbf y}$ to be the morphism $(\tau_0 \colon \tau_1) \mapsto (y_0\tau_1 \colon y_1\tau_1 \colon y_1^2 \tau_0 \colon y_0^2 \tau_0)$ for a fixed representation $(y_0,y_1)$ for the point $\mathbf y \in \Ps^1(K)$. Note that $\rho_{\mathbf y}$ restricts to an isomorphism $\mathbb A^1 \to V_{\mathbf y}$ where $\mathbb A^1 \subseteq \Ps^1$ is given by $\tau_0 \ne 0$. Hence we can compute the relevant integral locally on this open with variable $\tau=\tau_1/\tau_0$, using the invertible section $\tau_0$. We also write $y_{\textup m, \nu}$ for one of the $y_i$ for which $|y_i|_\nu$ is maximal.
\begin{align*}
\tau_\nu(V_{\mathbf y}) & = \int_{V_{\mathbf y}(K_\nu)} \omega_{V_{\mathbf y},\nu}  = \int_{\mathbb A^1(K_\nu)} \rho_{\mathbf y}^*\omega_{V_{\mathbf y},\nu}\\
& = \int_{\mathbb A^1(K_\nu)} \sup\left\{\left|y_0\tau\right|_\nu, \left|y_1\tau\right|_\nu, \left|y^2_1\right|_\nu, \left|y^2_0\right|_\nu \right\}^{-2[K_\nu \colon \Q_p]}\ d\tau\\
& = |y_{\textup{m},\nu}|_\nu^{-4[K_\nu \colon \Q_p]} \int_{\mathbb A^1(K_\nu)} \sup\left\{\left|\frac{\tau}{y_{\textup{m},\nu}}\right|_\nu, 1\right\}^{-2[K_\nu \colon \Q_p]}\ d\tau\\
& = |y_{\textup{m},\nu}|_\nu^{-3[K_\nu \colon \Q_p]} \int_{\mathbb A^1(K_\nu)} \sup\left\{\left|\tau\right|_\nu, 1\right\}^{-2[K_\nu \colon \Q_p]}\ d\tau\\
& = |y_{\textup{m},\nu}|_\nu^{-3[K_\nu \colon \Q_p]} \tau_\nu(\Ps^1_K).
\end{align*}
Hence after taking the product over all places we end up with
\[
\tau(V_{\mathbf y}) = \frac1{H_{\Ps^1}(\mathbf y)^3} \tau(\Ps^1_K).
\]
The relation between the expected constants of each $V_{\mathbf y}$ and the projective line follows directly from \eqref{eq:expectedconstant}.
\end{proof}

\begin{Remark}
	Let $\Delta_K,\omega_K, h_K, R_K, r$ and $s$ denote the discriminant, the number of roots of unity in $K$, the size of the class group, the regulator and the number of real and complex embeddings of $K$, respectively. Schanuel \cite{Schanuel} proved that
	$$
	c_{\Ps^1, K} = 2^{r+s-1}\frac{2^{2r} (2\pi)^{2s} h_K R_K}{\left| \Delta_K\right| \omega_K \zeta_K(2)}
	$$
	and these constants agree with the predicted constants $c(\Ps^1_K)$.
\end{Remark}

We will now state the result of the count of points on bounded height on $V$. This result reflects the geometry of $V$ since the constant will be the sum over all $c(V_{\mathbf y})$. To simplify notation we use the height zeta function of a variety $A$ over a field $K$, which is
\begin{equation}\label{eq:heightzetafunction}
Z_K(A,s) = \sum_{\mathbf x \in A(K)} H(\mathbf x)^{-s}.
\end{equation}
By abuse of notation we sometimes use $Z_K(B,s)$ for a subset $B \subseteq A(K)$ for the sum restricted to those points.
\begin{Theorem}\label{thm:countonV}
	We have for $d> 0$
	$$
	N_{V,K}(B) = c_{\Ps^1,K}\left(Z_K(\Ps^1,3)+1\right) B^2 + O\left(\hat{c}_{V,K} B^{2-1/d}\right).
	$$
	Here
	$$
	\hat{c}_{V,K} = \frac{Z_K\left(\Ps^1,3-2/d\right)}{\zeta_K\left(2-1/d\right)}
	$$
	and the implied constant only depends on the degree $d$ of $K$.
	In the case $d=2$ the error term needs to be replaced by $O(\hat{c}_{V,K} B^{3/2} \log B)$ and $\hat{c}_{V,K} = \zeta_K\left(3/2\right)^{-1}c_{\Ps^1, K}$. Furthermore, this result agrees with the conjecture by Batyrev and Tschinkel.
\end{Theorem}

Note that this implies that the set consisting of $V_0$ and the $V_{\mathbf y}$ form an asymptotic arithmetic fibration in the terminology of \cite{BaTsch}, although not each element is contained in a fibre of the natural fibration $\varphi \colon X \to Y$.

\section{Rational points on the cubic surface}

We will now provide a proof for the count of the rational points on the open subset $V$ of the cubic surface given by $t_0^2t_2 = t_1^2t_3$ over all number fields $K$. 

In the following we construct a parametrisation of the rational points $(t_0 \colon t_1\colon t_2\colon t_3) \in V$ which does not work if $t_i = 0$ for some $0\leq i \leq 3$.

We denote by $U :=V \setminus\{t_0t_1t_2t_3 = 0\}$ the set of points with non-vanishing coordinates. Points with vanishing coordinates lie on one of the lines $\{t_0 = t_3 =0 \}$, $\{t_1 = t_2 =0 \}$, $\{t_2 = t_3 =0 \}$ that each contribute with $N_{\Ps^1, K}(B) = c_{\Ps^1, K} B^2 + O(B \log B)$ to the count, according to \cite{Schanuel}. We will spend the rest of this chapter establishing the count on $U(K)$. The asymptotic for the number of points of bounded height on $U$ was already done in \cite[Section~4.5]{BaTsch}. Our approach will yield explicit error terms, which we will need later on.

\subsection{Parametrisation of rational points}\label{sec:Para}
Working over number fields prevents us from having unique factorisation in the ring of integers. Therefore, we turn our attention to ideals generated by the coordinates of the rational points. Let $\mathbf t = (t_0,t_1,t_2,t_3) \in U(K)$. We denote by $\mathfrak J(\mathbf t) = (t_0 \mathcal O_K, t_1 \mathcal O_K, t_2 \mathcal O_K, t_3 \mathcal O_K)$ the smallest ideal that contains all principle ideals generated by the coordinates of $\mathbf t$. We fix, once and for all, a system $\mathcal D_K$ of integral ideals which represent the $h_K$ different ideal classes of $\mathcal O_K$. We choose the representatives to be minimal with respect to the ideal norm. Let $C \in \mathcal D_K$ be the representative of the ideal class such that $[\mathfrak J(\mathbf t)] = [C]$. By multiplying with a suitable element from $K^\times$, we can choose a representative of $\mathbf t \in (\mathcal O_K\setminus \{0\})^4$ with $\mathfrak J(\mathbf t) = C$.
\begin{Lemma}\label{lem:ParaOfIntId}
	For all $\mathbf t \in U(K)$ with $\mathfrak J(\mathbf t) = C$ there exist integral non-zero ideals $\mathfrak a_0, \mathfrak a_1, \mathfrak a_2, \mathfrak a_3 \subseteq \mathcal O_K$ such that the coprimality conditions $(\mathfrak a_0, \mathfrak a_1) = (\mathfrak a_2, \mathfrak a_3) = \mathcal O_K$ hold and we have the parametrisation of integral ideals 
	\begin{align*}
	t_0 \mathcal O_K &= C\mathfrak a_0 \mathfrak a_3, &&t_2\mathcal O_K =C \mathfrak a_1^2 \mathfrak a_2,\\
	t_1 \mathcal O_K &= C \mathfrak a_1 \mathfrak a_3, && t_3\mathcal O_K = C \mathfrak a_0^2 \mathfrak a_2.
	\end{align*} 
\end{Lemma}
\begin{proof}
	Let $\mathfrak a_2, \mathfrak a_3$ be nonzero ideals of $\mathcal O_K$ such that $(t_0\mathcal O_K,t_1\mathcal O_K) = C \mathfrak a_3$ and $(t_2\mathcal O_K,t_3\mathcal O_K) = C \mathfrak a_2$. It follows that $(\mathfrak a_2, \mathfrak a_3) = \mathcal O_K$. Further, there are nonzero ideals $\mathfrak a_0, \mathfrak a_1$ of $\mathcal O_K$ with $t_0 \mathcal O_K = C \mathfrak a_0 \mathfrak a_3$ and $t_1 \mathcal O_K =C \mathfrak a_1 \mathfrak a_3$. In the same sense, let $\mathfrak b_2, \mathfrak b_3$ be such that $t_2\mathcal O_K = C \mathfrak b_2 \mathfrak a_2$ and $t_3\mathcal O_K = C\mathfrak b_3 \mathfrak a_2$. We use the describing equation and obtain
	$$
	\mathfrak a_0^2 \mathfrak b_2 = \mathfrak a_1^2 \mathfrak b_3.
	$$
	By construction $(\mathfrak a_0,\mathfrak a_1) = \mathcal O_K$, which yields $\mathfrak a_0^2 = \mathfrak b_3$ and $\mathfrak a_1^2 = \mathfrak b_2$.
\end{proof}
By fixing the ideal class of one of the integral ideals from Lemma \ref{lem:ParaOfIntId}, say $\mathfrak a_0$, we obtain a parametrisation of rational points in the following way. There is a unique $C_0\in \mathcal D_K$, such that $[\mathfrak a_0] = [C_0^{-1}]$. The product $\mathfrak a_0 C_0$ is again integral, so there is a $y_0 \in \mathcal O_K\setminus \{0\}$ with $y_0\mathcal O_K = \mathfrak a_0 C_0$ and we obtain
$$
t_0\mathcal O_K = C y_0 C_0^{-1} \mathfrak a_3,
$$
which indicates $[\mathfrak a_3] = [C_0 C^{-1}] =: [C_3^{-1}]$. The choice of $y_0$ and $t_0$ determines a unique $y_3 \in \mathcal O_K\backslash \{0\}$ with $y_3\mathcal O_K = \mathfrak a_3 C_3$, such that 
$$
t_0 = y_0 y_3.
$$
In the same way we also find $[\mathfrak a_1] = [C_0^{-1}] =: [C_1^{-1}]$ and $[\mathfrak a_2] = [C^{-1} C_0^2]=: [C_2^{-1}]$. The elements $y_i \in \mathcal O_K\backslash \{0\}$ with $y_i \mathcal O_K = \mathfrak a_i C_i$ for $i=1,2$ are uniquely determined by $t_1/y_3$ and $t_2/y_1^2$ respectively. We obtain the parametrisation of rational points
\begin{equation}\label{eq:parametrisation}
\begin{split}
t_0  = y_0 y_3,\hspace{4cm} t_2 =y_1^2 y_2,\\
t_1  = y_1 y_3,\hspace{4cm}  t_3 = y_0^2 y_2.
\end{split}
\end{equation}
\begin{Definition}\label{DefMCC}
	For $(C,C_0) \in \mathcal D_K^2$, let $M(C,C_0)$ be the set of $\mathbf y \in (\mathcal O_K\backslash \{0\})^4$ such that 
	\begin{enumerate}\label{DefinitionOfMC}
		\item $y_k \in C_k$ for $k=0,\dots, 3$ with $C_1 = C_0$, $C_2=C C_0^{-2}$ and $C_3= C C_0^{-1}$, and
		\item the ideals $y_k C_k^{-1}$ satisfy $(y_0 C_0^{-1},y_1 C_1^{-1}) = (y_2 C_2^{-1} ,y_3 C_3^{-1}) = \mathcal O_K$.
	\end{enumerate}
\end{Definition}
We have found an surjective map 
$$
\phi \colon \bigcup_{(C,C_0) \in \mathcal D_K^2} M(C, C_0) \to U(K)
$$
that is far from being injective. The choice of $y_1$ in the construction above is unique up to multiplication by elements in $\mathcal O_K^\times$. We want to find a subset $A \subseteq M(C,C_0)$ such that the preimage of $\mathbf t$ under $\phi$ is finite in $A$. Therefore, we take $\mathbf y \in M(C,C_0)$ and compare $\phi(\mathbf y)$ to $\phi(\mathbf \zeta \mathbf y)$ where $\mathbf \zeta \in \left(\mathcal O_K^\times\right)^4$ to trace the impact of units. Let $\mathcal O_K^\times = \mu_K \oplus \mathcal F$ where $\mu_K$ are the roots of unity in $K$ and $\mathcal F$ is the free abelian group of rank $r+s-1$ according to Dirichlet's unit theorem. We obtain the following result
\begin{Lemma}\label{lem:NiceMap}
	Let $R_1 \subseteq (K^\times)^2$ be a system of representatives for the orbits of the action of $\mathcal F$ on $(K^\times)^2$ given by scalar multiplication and let $\mathcal R := R_1 \times R_1 \subseteq K^4$. The map
	$$
	\phi \colon \mathcal R \cap \bigcup_{(C,C_0) \in \mathcal D_K^2} M(C,C_0) \to U(K)
	$$
	given by \eqref{eq:parametrisation} is $\omega_K^2$-to-1.
\end{Lemma}
We are left with establishing the height condition on the parametrisation of the rational points. Let $\mathbf y \in M(C,C_0)$ then
\begin{equation}\label{HeightProd}
H_K(\phi(\mathbf y)) = \mathfrak{N} \mathfrak{J}(\phi(\mathbf y))^{-1} \prod_{v \in M_K^{\infty} }\max_{j=0,1}\left\{\left|y_jy_3\right|_\nu, \left|y_j^2y_2\right|_\nu \right\},
\end{equation}
where $\mathfrak J(\phi(\mathbf y)) = C$ by construction. We define

\begin{equation}\label{eq:DefRB}
\mathcal R (B) := \left\{\mathbf y \in \mathcal R\colon \prod_{\nu \in M_K^\infty}\max_{j=0,1}\left\{\left|y_jy_3\right|_\nu, \left|y_j^2y_2\right|_\nu \right\} \leq B \right\}.
\end{equation}
Together with Lemma \ref{lem:NiceMap} we obtain the following expression for the counting function of $U$.
\begin{Proposition}\label{prop:SectionParaProp}
	Let $M(C,C_0)$ be as in Definition \ref{DefinitionOfMC} and let $\mathcal R(B)$ be as in \eqref{eq:DefRB}. Then $M(C,C_0) \cap \mathcal R(B)$ is finite for all $B>0$ and $(C,C_0) \in \mathcal D_K^2$. Furthermore, we have
	$$
	N_{U,K}(B) = \frac1{\omega^2_K} \sum_{ (C,C_0) \in \mathcal D_K^2} \#\big[M(C,C_0) \cap \mathcal R(\mathfrak N(C)B)\big].
	$$
\end{Proposition}
\subsection{Möbius inversion and lattice point counting}
For now we fix two integral representatives of ideal classes $C, C_0 \in \mathcal D_K$ and rewrite the cardinality of the intersection of sets in Proposition \ref{prop:SectionParaProp} as sums to
\begin{equation}\label{eq:setassum}
\#\big[M(C,C_0) \cap \mathcal R(B)\big] = \sum_{\substack{(y_0,y_1) \in R_1 \\ y_j \in C_j \text{ for }j =0,1 \\ (y_0 C_0^{-1}, y_1C_1^{-1}) =\mathcal O_K}} \sum_{\substack{(y_2,y_3) \in R_1\\y_i \in \mathfrak aC_i \text{ for }i=2,3\\ (y_2 C_2^{-1}, y_3 C_3^{-1}) = \mathcal O_K \\ \prod_{v\in M_K^\infty} \max_{j} \{|y_jy_3|_\nu, |y_j^2y_2|_{\nu} \} \leq B}} 1.
\end{equation}
We will estimate the inner sum by a lattice point counting result by Frei and show that the first sum is equal to a height zeta function.

We will resolve one of the coprimality conditions by performing a Möbius inversion for integral ideals. For any real or complex embedding $\sigma_\nu$ of $K$ let $\kappa_\nu := \max_{j=0,1}\{|\sigma_\nu(y_j)|\}$ and $\kappa := \prod_{\nu \in M_K^\infty} \kappa_\nu$. We substitute this into the height condition of the inner sum in \eqref{eq:setassum} and obtain 
\begin{equation}\label{eq:NewHeight}
\prod_{\nu \in M_K^\infty}\max\left\{\kappa_\nu \left|y_3\right|_\nu, \kappa_\nu^2 \left|y_2\right|_\nu\right\}\leq B.
\end{equation}
Let $I_K$ be the set of integral ideals $\mathfrak a \subseteq \mathcal O_K$. We define the function $f \colon I_K \to \C$ by
$$
f(\mathfrak a) := \sum_{\substack{(y_2,y_3) \in R_1\\ y_i \in \mathfrak aC_i \text{ for }i=2,3\\(y_2C_2^{-1},y_3C_3^{-1})= \mathfrak a\\ \prod_{\nu \in M_K^\infty} \max \{\kappa_\nu|y_3|_\nu, \kappa_\nu^2|y_2|_\nu \} \leq B}} 1
$$
for $\mathfrak a \in I_K$. With this choice of $f$, the inner sum of \eqref{eq:setassum} equals $f(\mathcal O_K)$.
Note that $f$ vanishes on ideals of norm bigger than $B$ 
which means that we can define a function $g \colon I_K \to \C$ by taking the finite sum
$$
g(\mathfrak a) = \sum_{\mathfrak b \in I_K} f(\mathfrak a \mathfrak b) =  \sum_{\substack{(y_2,y_3) \in R_1\\ y_i \in \mathfrak aC_i \text{ for }i=2,3\\ \prod_{\nu \in M_K^\infty} \max \{\kappa_\nu|y_3|_\nu, \kappa_\nu^2|y_2|_\nu \} \leq B}} 1.
$$
We compute $f(\mathcal O_K)$ using the Möbius inversion formula for integral ideals and obtain
$$
f(\mathcal O_K) = \sum_{\mathfrak d \in I_K} \mu(\mathfrak d) g(\mathfrak d) = \sum_{\mathfrak d \in I_K} \mu(\mathfrak d)   \sum_{\substack{(y_2,y_3) \in R_1\\ y_i \in \mathfrak dC_i \text{ for }i=2,3\\ \prod_{\nu \in M_K^\infty} \max \{\kappa_\nu|y_3|_\nu, \kappa_\nu^2|y_2|_\nu \} \leq B}} 1.
$$
We estimate the inner sum by a variation of a lattice point counting result by Frei \cite[Lemma 5.2]{Frei}:
\begin{Lemma}\label{lem:FreiLattice}
	Given constants $\kappa_\nu >0$ let $\kappa := \prod_{\nu \in M_K^\infty} \kappa_\nu$. Let $\mathfrak a_2, \mathfrak a_3 \neq \{0\}$ be fractional ideals of $K$, and $R_1$ a system of representatives for the orbits of $(K^\times)^2$ under the action of $\mathcal F$ by scalar multiplication. Define
	$$
	M_1(B) :=  \#\left\{(y_2,y_3) \in(\mathfrak a_2 \times \mathfrak a_3) \cap R_1 \colon \eqref{eq:NewHeight} \text{ holds} \right\}.
	$$
	Then $M_1(B)$ is finite and 
	$$
	\left|M_1(B) - 2^{r+s-1}\frac{ 2^{2r} (2\pi)^{2s} R_K}{|\Delta_K| \kappa^3 \mathfrak{Na}_2 \mathfrak{Na}_3} B^2\right| \ll_{d} \frac{B^{2-1/d} \max\{\kappa \mathfrak{Na}_3, \kappa^2 \mathfrak{Na}_2 \}^{1/d}}{\kappa^3 \mathfrak{Na}_2\mathfrak{Na}_3},
	$$
	for all $B>0$. The implicit $O$-constant only depends on the degree of the number field.
\end{Lemma}
\begin{proof}
	We apply the proof of \cite[Lemma 5.2]{Frei} verbatim to this case where two ideals $\mathfrak a_2, \mathfrak a_3$ are considered. Note that the implicit constant in Frei's Lemma comes from \cite[Lemma 4.4]{Frei} where the implicit constant depends on invariants of the basic set $S_F^{d\ast}(1)$ but not on invariants of the number field $K$.
\end{proof}
We note that $M_1(B)$ from Lemma \ref{lem:FreiLattice} equals the inner sum of $f(\mathcal O_K)$. We set $\mathfrak{a}_2 = \mathfrak d C_2= \mathfrak d C C_0^{-2}$ and $\mathfrak{a}_3 = \mathfrak d C_3 = \mathfrak d C C_0^{-1}$ and substitute everything back into \eqref{eq:setassum}. This yields the following expression
\begin{equation*}
\begin{aligned}
\#[M(C,C_0)\cap \mathcal R(B)] =& \sum_{\mathfrak d \in I_K} {\mu(\mathfrak d)} \sum_{\substack{(y_0,y_1) \in R_1 \\ y_j \in C_j \text{ for }j=0,1\\(y_0C_0^{-1}, y_1C_1^{-1})= \mathcal O_K\\ \kappa \leq \sqrt{B}}} \Bigg(2^{r+s-1}\frac{ 2^{2r} (2\pi)^{2s} R_K \mathfrak N (C_0)^3B^2}{{\mathfrak N(\mathfrak d)^2}|\Delta_K| \kappa^3 \mathfrak{N}(C)^2}\\ &+ O\left(\frac{B^{2-1/d}\mathfrak N(C_0)^{3-2/d}}{\kappa^{3-2/d}\mathfrak N(\mathfrak d)^{2-1/d} \mathfrak N(C)^{2-1/d}}\right)\Bigg).
\end{aligned}
\end{equation*}
We replace the sum over $\mathfrak d$ by the reciprocal of the Dedekind zeta function evaluated at 2 in the main term and for the error term we replace it by the Dedekind zeta function evaluated at $2-1/d$.
\begin{Lemma}\label{lem:OmegaKto1}
	The canonical map
	\begin{align*}
	\psi \colon \bigcup_{C \in \mathcal D(K)}\left\{(y_0,y_1) \in R_1 \colon (y_0 \mathcal O_K, y_1 \mathcal O_K) = C \right\} &\to \Ps^{1}(K)\\
	(y_0,y_1) &\mapsto (y_0 \colon y_1)
	\end{align*}
	is $\omega_K$-to-1 and surjective.
\end{Lemma}
We use Lemma \ref{lem:OmegaKto1} together with the observation that $H_K(y_0,y_1) = \kappa \mathfrak N(C_0)^{-1}$ and obtain for the sum in Proposition \ref{prop:SectionParaProp}
\begin{equation}\label{eq:EndOfSectionSum}
\sum_{ (C,C_0) \in \mathcal D_K^2} \sum_{ \substack{(y_0,y_1) \in \Ps^1(K)\\ (y_0\mathcal O_K,y_1\mathcal O_K) = C_0\\ y_0y_1 \neq 0\\ H_K(y_0,y_1) \leq \sqrt{B}}}\left( \frac{c_{\Ps^1,K}}{H_K(y_0,y_1)^3} B^2 + O\left(\frac{1}{\zeta_K(2-1/d)H_K(y_0,y_1)^{3-2/d}}B^{2-1/d}\right)\right).
\end{equation}
There is no dependence on $C$ so we replace the sum over $C \in \mathcal D_K$ by multiplying the main term with $h_K$. Combing the two sums together, we obtain 
$$
Z_K\left(\Ps^1, 3, \sqrt{B} \right) -2 
$$
in the main term, where the additional argument in the height zeta function indicates that we only sum up to points of height $\sqrt B$. We deduct 2 from the height zeta function, because we exclude the contribution coming from points that have a vanishing coordinate, which are $(1,0)$ and $(0,1)$. For the error terms we obtain $Z_K(\Ps^1, 3-2/d, \sqrt{B})$.

In the main term, we extend the sum to infinity, so that we obtain the whole height zeta function. The tail gives a contribution of $O(B^{3/2})$ which can be seen by using dyadic summation. We treat the sum in the error term in the same way and receive $O(B^{2-1/d})$. In the special case $d=2$ the height zeta function $Z_K(\Ps^1,2)$ is not convergent, so we apply Abel's summation formula and obtain
$$
Z_K\left(\Ps^1,2,\sqrt{B}\right) \ll c_{\Ps^1,K} \log B.
$$
We combine \eqref{eq:EndOfSectionSum} with Proposition \ref{prop:SectionParaProp} and account for the three removed lines with an additional $3N_{\Ps^1,K}(B)$, as indicated earlier, in order to finally establish the third statement of Theorem \ref{thm:introductionV}.

\section{Count of rational points of the symmetric product}

Let $X$ be a quasi-projective $K$-variety and $m\geq1$, we define the $m$th symmetric product of $X$ as 
$$
\Sym^m X = X^m/ \operatorname{S}_m
$$
where $\operatorname{S}_m$ is the symmetric group on $m$ elements. The $K$-points of $\Sym^m X$ are representatives for Galois invariant $m$-tuple of $\bar K$-points on $X$ under permutation. There is a canonical projection map $\pi \colon X^m \to \Sym^m X$. Every point $\tilde{\mathbf{x}} \in \Sym^m X$ is the imagine of the projection of $m$ points $\mathbf x_1,\dots,\mathbf x_m \in X(\bar \Q)$ under $\pi$.

Narrowing down our attention once again to the smooth locus $V$ of the cubic surface $W\subseteq \Ps^3$ defined by $t_0^2t_2 = t_1^2t_3$, we observe that $\Sym^2 V$ has infinitely many rational points. By imposing a height condition on the rational points of $\Sym^2 V$ we can use Northcott's theorem to obtain a finite counting problem.

Let $H_K(\mathbf x)$ be the relative height for $\mathbf x \in V(K)$. Let $\mathbf x_1, \mathbf x_2 \in V(\bar \Q)$. We will see in Proposition~\ref{prop:adelicmetriconsym2} that the height function on $V$ induces a natural height function on $\Sym^2 V$. The height of a point $\pi(\mathbf x_1, \mathbf x_2) = \tilde{\mathbf x} \in \Sym^2 V$ is then
$$
H(\tilde{\mathbf x}) = \big(H_{K}(\mathbf x_1)H_{K}(\mathbf x_2)\big)^{1/d},
$$
where $K$ is the smallest extension of $\Q$ that contains $\mathbf x_1, \mathbf x_2$ and $d = [K \colon \Q]$.
We want to find an asymptotic formula for the counting function
$$
N_{\Sym^2 V,\Q}(B) :=\# \left\{\mathbf x \in \Sym^2 V(\Q) \colon H(\mathbf x) \leq B \right\},
$$
as $B$ tends towards infinity.

\begin{Theorem}\label{thm:sym2}
Consider the symmetric square $\Sym^2 V$ for the surface $V$ defined over the rationals and the fibration $\varphi \colon V \to \Ps^1$ into lines from Definition~\ref{dfn:fibrationmap}. Define
\[
\mathcal Z= \{ \pi(\mathbf x_1, \mathbf x_2) \in \Sym^2 V(\Q) \colon \mathbf x_1, \mathbf x_2 \in V(\bar\Q), \varphi(\mathbf x_1)=\varphi(\mathbf x_2)\}.
\]
\begin{enumerate}
\item[(a)] As $B \to \infty$ we have
\[
N_{\mathcal Z}(B) = cB^3 + O\left(B^2 \log B\right)
\]
where the constant $c$ is given by
$$
c_{\Sym^2 \Ps^1, \Q}\left( Z_{\Q}\left(\Ps^1,9\right)+1\right),
$$
with $c_{\Sym^2 \Ps^1, \Q} = 4 \zeta(3)^{-1}$ according to \cite[Theorem 3]{Schmidt}.
\item[(b)] As $B \to \infty$ we have
\[
N_{\Sym^2V(\Q)\setminus\mathcal Z}(B) \ll B^2\log B.
\]
\item[(c)] We conclude
\[
N_{\Sym^2 V,\Q}(B) = cB^3 + O\left( B^2 \log B\right).
\]
\end{enumerate}
\end{Theorem}

We will see in the next section that $\mathcal Z$ is actually the set of rational points on a closed subvariety $\mathbb V \subseteq \Sym^2 V$ and hence a thin set of type I.

\begin{Remark}\label{rmk:rat}
	The set of points $\pi(\mathbf x, \mathbf y) \in \Sym^2 V$ with $\mathbf x, \mathbf y \in V(\Q)$ are spread over parts (a) and (b) of Theorem~\ref{thm:sym2}. We point out that the count of these points is
	$$
	\# \left\{\mathbf x,\mathbf y \in V(\Q)\colon H_{\Q}(\mathbf x)H_{\Q}(\mathbf y) \leq B\right\} \sim 2c^2_{V,\Q} B^2 \log B \left(1 + o(1)\right),
	$$
	with $c_{V,\Q}$ the constant from Theorem~\ref{thm:introductionV}.

	One can use this to show that $N_{\Sym^2V(\Q)\setminus\mathcal Z}(B)$ in Theorem~\ref{thm:sym2}(b) also has a lower bound of order $B^2 \log B$. This is the order of growth predicted by Manin's conjecture for the variety $\Sym^2 V$.

	We can also give an explicit lower bound $cB^2 \log B$ for the constant $c$ predicted by Batyrev and Tschinkel for $\Sym^ V\backslash \mathbb V$. However our current approach yields a strictly larger upper bound. So our techniques are currently not strong enough to establish the required asymptotic behaviour.
\end{Remark}

\subsection{The Batyrev--Tschinkel conjecture for $\Sym^2 V$}

We will interpretate the count of $\Q$-points on bounded height on $\Sym^2 V$ in terms of the Batyrev--Tschinkel conjecture. First we will recall how an adelically metrized line bundle on a scheme $X$ gives rise to one on $\Sym^m X$.

\begin{Proposition}[{\cite[Proposition~1.33]{RudulierThesis}}]\label{prop:adelicmetriconsym2}
Let $K$ be a number field, $m\geq 1$ an integer and $(\mathcal L,\|.\|_\nu)$ an adelically metrized line bundle on a quasi-projective $K$-scheme $X$. Let $\pi_i \colon X^m \to X$ the projection on the $i$th coordinate. The line bundle
\[
\mathcal L^{\boxtimes m} := \bigotimes \pi^*_i \mathcal L
\]
admits an adelic metric constructed by taking the tensor product of the pull back metric on each factor. This adelically metrized line bundle is symmetric and descends along the quotient map $\pi \colon X^m \to \Sym^m X$ to a line bundle $\mathcal L^{(m)}$ with the unique adelic metric $\{\|.\|'\}_\nu$ which satisfies
\[
\|s(\pi(\alpha_1,\ldots, \alpha_m))\|'_\nu = \|(\pi^*s)(\alpha_1,\ldots, \alpha_m)\|_\nu
\]
for any section $s \in \Gamma(\Sym^m X,\mathcal L^{(m)})$ non-zero at $\pi(\alpha_1,\ldots, \alpha_m)$.
\end{Proposition}

So we have an adelic metric on the line bundle $\mathcal L^{(2)}$ on $\Sym^2 V$ and we will count points which are bounded with respect to the associated height. Similar to the count on $V$ coming from the presence of infinitely many lines, the count on $\Sym^2 V$ is explained by the presence of infinitely many planes. Let us first explain where these planes come from.

\begin{Definition}
Define $V_{0,0}$ as the proper closed subscheme $\Sym^2 V_0 \to \Sym^2 V$. Similarly for any point $\mathbf y \in \Ps^1(\Q)$ we have $V_{\mathbf y, \mathbf y} := \Sym^2 V_{\mathbf y}$.
\end{Definition}

Since $V_0$ and $V_{\mathbf y}$ are respectively a projective and an affine line we see that $V_{0,0}$ and $V_{\mathbf y, \mathbf y}$ are respectively a projective and an affine plane.

\begin{Proposition}\label{prop:sym2VinBTterms}
Theorem~\ref{thm:sym2} agrees with the expected result by Batyrev and Tschinkel. This shows that the point count on $\Sym^2 V$ is explained by its strongly saturated subvarieties $V_{0,0}$ and $V_{\mathbf y, \mathbf y}$. 
\end{Proposition}

\begin{proof}
Using the procedure laid out by Batyrev and Tschinkel we first reduce to two weakly saturated subvarieties of $\Sym^2 V$. The first subvariety is $V_{0,0}$, which is also strongly saturated. Even more it is isomorphic to the symmetric square of the $\Ps^1$ with the corresponding induced adelic measure coming from the usual adelic measure on $\mathcal O(1)$. Hence it is known due to Le Rudulier \cite{RudulierThesis} that
\[
N_{V_{0,0},\Q}(B) \sim c_{\Sym^2 \Ps^1,\Q} B^3
\]
as predicted by Batyrev and Tschinkel, or Manin \cite{ManinCon} and Peyre \cite{PeyreConstant} since $\Sym^2 \Ps^1$ is a Fano variety.

The second weakly saturated subvariety $\mathbb V \subseteq \Sym^2 V$ is one which makes the $V_{\mathbf y, \mathbf y}$ into a family of planes over $\Ps^1$. It can be constructed as the fibred product
\[
\begin{tikzcd}
\mathbb V \arrow[r] \arrow[d] & \Sym^2 V \arrow[d,"\Sym^2 \varphi"]\\
\Ps^1 \arrow[r, "\Delta"] & \Sym^2 \Ps^1
\end{tikzcd}
\]
along the diagonal morphism $\Ps^1 \to \Sym^2 \Ps^1$. One can use this construction to show that the set $\mathcal Z$ from Theorem~\ref{thm:sym2} is precisely $\mathbb V(\Q)$.

This morphism $\mathbb V \to \Ps^1$ is an $\mathcal L^{(2)}$-primitive fibration in the words of Definition~2.4.2 in \cite{BaTsch}. It follows from Theorem~\ref{thm:sym2} that
\[
N_{\mathbb V,\Q}(B) \sim c_{\Sym^2 \Ps^1,\Q} Z(\Ps^1,9) B^3.
\]
We will show that this count is also predicted by Batyrev and Tschinkel.

First note that the fibres $V_{\mathbf y, \mathbf y}$ of $\mathbb V \to \Ps^1$ are strongly saturated. As a scheme they are isomorphic to $\Sym^2 \mathbb A^1 \cong \mathbb A^2$, and the line bundle $\mathcal L^{(2)}$ restricts to $\mathcal O(1)$ for each $\mathbf y$. Only the adelic metric depends on $\mathbf y$. First of all, this shows that the natural compactification of any $V_{\mathbf y, \mathbf y}$ is $X_{\mathbf y, \mathbf y} := \Sym^2 X_{\mathbf y} \cong \Sym^2 \Ps^1 \cong \Ps^2$. The following constants only depend on this non-singular compactification with its line bundle and are hence known to be
\[
a(X_{\mathbf y, \mathbf y})=a(\Ps^2)= 3, \quad b(X_{\mathbf y, \mathbf y})=b(\Ps^2)= 1,
\]
\[
\gamma(X_{\mathbf y, \mathbf y})=\gamma(\Ps^2)= 1 \quad \text{ and } \quad \delta(X_{\mathbf y, \mathbf y})=\delta(\Ps^2)= 1.
\]
We now address the last constant $\tau(X_{\mathbf y, \mathbf y})$ which will vary as the adelic metric depends on $\mathbf y$. We compute its local factors using the isomorphism $\rho_{\mathbf y}$ in the proof of Proposition~\ref{prop:VinBTterms}, which induces an isomorphism $V_{\mathbf y, \mathbf y} \to \Sym^2 \mathbb A^1$. The morphism $\mathbb A^2 \to \Sym^2\mathbb A^1$ given by
\[
(t_1,t_2) \mapsto \text{ the roots of } T^2-t_1T+t_2
\]
is also an isomorphism. We will use the composition of these two isomorphisms to pullback the integral on $X_{\mathbf y, \mathbf y}$ to one on $\Ps^2$.
\[
\tau_\nu(X_{\mathbf y, \mathbf y}) = \int_{V_{\mathbf y, \mathbf y}(K_\nu)} \omega_{X_{\mathbf y, \mathbf y},\nu} = \int_{\mathbb A^2(K_\nu)} \prod_{\substack{\text{roots $\tau$ of}\\ T^2-t_1T+t_2}} \frac{dt_1\ dt_2}{\sup\left\{\left|y_0\tau\right|_\nu, \left|y_1\tau\right|_\nu, \left|y^2_1\right|_\nu, \left|y^2_0\right|_\nu \right\}^{3[K_\nu \colon \Q_p]}}.
\]
As in the proof of Proposition~\ref{prop:VinBTterms}, define $y_{\textup m, \nu}$ to be one of the $y_i$ for which $|y_i|_\nu$ is maximal. Then we get
\begin{align*}
\tau_\nu(X_{\mathbf y, \mathbf y}) & = |y_{\textup m, \nu}|^{-12 [K_\nu \colon \Q_p]} \int_{\mathbb A^2(K_\nu)} \prod_{\substack{\text{roots $\tau$ of}\\ T^2-t_1T+t_2}} \frac{dt_1\ dt_2}{\sup\left\{\left|\frac{\tau}{y_{\textup{m},\nu}}\right|_\nu, 1\right\}^{3[K_\nu \colon \Q_p]}}\\
& = |y_{\textup m, \nu}|^{-9 [K_\nu \colon \Q_p]} \int_{\mathbb A^2(K_\nu)} \prod_{\substack{\text{roots $\tau$ of}\\ T^2-t_1T+t_2}} \frac{dt_1\ dt_2}{\sup\left\{\left|\tau\right|_\nu, 1\right\}^{3[K_\nu \colon \Q_p]}}\\
& = |y_{\textup m, \nu}|^{-9 [K_\nu \colon \Q_p]} \tau_\nu(\Sym^2 \Ps^1).
\end{align*}
Here we have used the change of variables $t_1 \mapsto t_1/y_{\textup m, \nu}$ and $t_2 \mapsto t_2/y^2_{\textup m, \nu}$, and the computation of $\tau_\nu$ for the symmetric square of $\Ps^1$ by Le Rudulier \cite{RudulierThesis}. And we conclude that
\[
\tau(X_{\mathbf y, \mathbf y}) = \frac1{H_{\Ps^1}(\mathbf y)^9} \tau(\Sym^2 \Ps^1).
\]
This proves the lemma.
\end{proof}

\subsection{A dichotomy of rational points}

All statements in Theorem~\ref{thm:sym2} follow from the following lemma.

\begin{Lemma}\label{lem:sym2}
Define
\[
\mathcal Z'= \{ \pi(\mathbf x, \bar{\mathbf x}) \in \Sym^2 V(\Q) \colon \mathbf x \in V(\bar\Q), \deg \mathbf x = 2, \varphi(\mathbf x) \in \Ps^1(\Q)\}.
\]
\begin{enumerate}
\item[(a)] As $B \to \infty$ we have
\[
N_{\mathcal Z'}(B) = cB^3 + O\left(B^2 \log B\right),
\]
where the constant $c$ is given by
$$
c_{\Sym^2 \Ps^1, \Q}\left( Z_{\Q}\left(\Ps^1,9\right)+1\right).
$$
\item[(b)] As $B \to \infty$ we have
\[
N_{\Sym^2V(\Q)\setminus\mathcal Z'}(B) \ll B^2\log B.
\]
\end{enumerate}
\end{Lemma}

These statements imply Theorem \ref{thm:sym2}, since the set $\mathcal Z'$ is contained in the set $\mathcal Z$ from Theorem~\ref{thm:sym2}.
\begin{proof}
Let $\tilde{\mathbf x} = \pi(\mathbf x_1,\mathbf x_2) \in \Sym^2 V(\Q)$ be a rational point then $\sigma(\tilde{\mathbf x}) = \tilde{\mathbf x}$ for all $\sigma \in \Gal(\bar \Q/ \Q)$. This means that either $\mathbf x_1, \mathbf x_2 \in V(\Q)$, or $\mathbf x_1 \in V(\bar \Q)$ is a point of degree 2 and $\mathbf x_2$ is its conjugate under the action of the non-trivial element of the Galois group $\Gal(\Q(\mathbf x_1)/\Q)$. We rewrite the counting function
\begin{equation}\label{SymAusgang}
\begin{split}
N_{\Sym^2 V,\Q}\left(B\right) =& \frac12\# \left\{\mathbf x,\mathbf x' \in V(\Q)\colon \mathbf x \neq \mathbf x', \hspace{2mm} H_{\Q}(\mathbf x)H_{\Q}(\mathbf x') \leq B\right\} \\&+ \frac12\# \left\{\mathbf x\in V(\bar \Q)\colon [\Q(\mathbf x) \colon \Q] = 2, \hspace{2mm} H_{\Q(\mathbf x)}(\mathbf x) \leq B \right\} + O(B\log B).
\end{split}
\end{equation}
We have excluded points of the shape $(\mathbf x, \mathbf x)$ for $\mathbf x \in V(\Q)$ from the first set since their contribution is small.
We will refer to points in the first set as rational points of type 1 and accordingly to points in the second set as rational points of type 2. For the rational points of type 2 we used the fact that the height is invariant under the action of elements of the Galois group.

We start by giving a count of the rational points of type 1 which are contained in the statement of Lemma \ref{lem:sym2} (b). After that we will split the the type 2 points into two sets and prove statement (a) of Lemma \ref{lem:sym2}. At last we will show that the second set of type 2 points has the desired contribution to complete the proof of statement (b).

We rewrite the contribution coming from type 1 rational points and obtain
$$
\sum_{\substack{\mathbf x \in V(\Q)\\ H_{\Q}(\mathbf x) \leq B}} \# \left\{\mathbf x' \in V(\Q) \colon H_{\Q}(\mathbf x') \leq \frac{B}{H_{\Q}(\mathbf x)} \right\}.
$$
We use Theorem~\ref{thm:countonV} to bound the cardinality of the set and denote by $c_{V,\Q}$ the constant of the main term. We obtain
\begin{equation}\label{eq:RatPointFirst1}
c_{V,\Q} B^2\sum_{\substack{\mathbf x \in V(\Q)\\ H_{\Q}(\mathbf x) \leq B}} \frac{1}{H_{\Q}(\mathbf x)^2} + O\left(B^{3/2}\log B\sum_{\substack{\mathbf x \in V(\Q)\\ H_{\Q}(\mathbf x) \leq B}} \frac{1}{H_{\Q}(\mathbf x)^{3/2}}\right).
\end{equation}
For both the main and the error term we need to estimate the partial height zeta function $Z_{\Q} (V,\delta,B)$ for $\delta =2$ and $\delta = 3/2$. Using Abel's summation formula we arrive at the following expression
	$$
	Z_{\Q}(V,\delta, B) = c_{V,\Q} B^{2-\delta} + \delta c_{V,\Q} \int_1^B u^{1-\delta}du + O\left(B^{3/2-\delta}\right).
	$$
Substituting this into \eqref{eq:RatPointFirst1} yields the total contribution coming from rational points of type 1 as claimed in Remark \ref{rmk:rat}.

In order to determine the contribution of the rational points of type 2 to the count we need to estimate the set
\begin{equation*}
\# \left\{\mathbf x \in V(\bar{\Q}) \colon \deg \mathbf x = 2, H_{\Q(\mathbf x)} (\mathbf x) \leq B \right\}.
\end{equation*}
We reuse ideas and statements of the proof of Theorem~\ref{thm:countonV}. There are three lines on $V(\bar \Q)$ on which points with vanishing coordinates lie, namely $\{t_0 = t_3 =0 \}, \{t_1 = t_2 =0\}, \{t_2 = t_3=0\}$. We will again denote by $U(\bar \Q)$ the set of points on $V(\bar \Q)$ with non-vanishing coordinates and use the parametrisation of rational points in the following way.

For a fixed degree 2 number field $K$ we recall the parametrisation of rational points on $U(K)$ that we met in Section \ref{sec:Para}. For $C, C_0\in \mathcal D_K$ and $\mathbf y \in \mathcal R \cap M(C,C_0)$ we have the parametrisation of rational points \eqref{eq:parametrisation}. We sum \eqref{eq:EndOfSectionSum} over all quadratic number fields and obtain
\begin{equation}\label{eq:sumdeg2sym2}
\sum_{[K \colon \Q] =2} \frac{1}{\omega_K^2}\sum_{ (C,C_0) \in \mathcal D_K^2} \sum_{\substack{(y_0,y_1) \in R_1 \\ y_j \in C_j \text{ for }j =0,1 \\ (y_0 C_0^{-1}, y_1C_1^{-1}) =\mathcal O_K}} \sum_{\substack{(y_2,y_3) \in R_1\\y_i \in C_i \text{ for }i=2,3\\ (y_2 C_2^{-1}, y_3 C_3^{-1}) = \mathcal O_K \\ H_K(y_0y_3, y_1y_3, y_1^2 y_2, y_0^2 y_2) \leq B }} 1.
\end{equation}
The dominant contribution to the main term comes from the points $(y_0,y_1) \in R_1 \cap \Q^2$. Therefore, we divide the sum into two parts and define the sets
$$
U_1 = \bigcup_{K} \left\{(y_0,y_1) \times (y_2,y_3) \colon (y_0,y_1) \in R_1 \cap \Q^2, (y_2, y_3) \in R_1, \deg(y_2,y_3) = 2 \right\},
$$
and $U_2$ which is the complement of $U_1$ in $U(K)$. In the following we show that $U_1$ dominates the count and $U_2$ contributes solely to the error term.

There is a map from $R_1 \to \Ps^1(K)$ that we met in Lemma \ref{lem:NiceMap}. We obtain $(y_0\colon y_1) \in \Ps^1(\Q)$ and so $C_0 = C_1 = \mathcal O_K$ and $C_2 = C_3 = C$. The only part of the sum in \eqref{eq:sumdeg2sym2} that is relevant for this analysis is the one where $C_0$ is fixed to be $\mathcal O_K$, so we obtain 
	$$
	\sum_{[K \colon \Q] =2}\frac{1}{\omega_K} \sum_{ C \in \mathcal D_K} \sum_{\substack{(y_0\colon y_1)\in \Ps^1(\Q)\\y_0y_1 \neq 0}} \sum_{\substack{(y_2,y_3) \in R_1\\ [\Q(y_2\colon y_3) \colon \Q] = 2\\y_i \in C \text{ for }i=2,3\\ (y_2\mathcal O_K, y_3 \mathcal O_K) = C \\ H_K(y_0y_3, y_1y_3, y_1^2 y_2, y_0^2 y_2) \leq B}} 1,
	$$
	where $\Q(y_2\colon y_3)$ is the smallest number field containg $y_2/y_3$. Note that the last sum combined with the sum over the integral ideal representatives of elements of the class group is the same as taking the sum over degree 2 points in $\Ps^1(K)$ with non-vanishing coordinates. Changing the order of summation yields
	\begin{equation}\label{eq:pointsQbarSym}
	\sum_{ (y_0 \colon y_1)\in \Ps^1 (\Q)} \sum_{\substack{(y_2\colon y_3) \in\Ps^1(\bar{\Q})\\ [\Q(y_2\colon y_3) \colon \Q] = 2\\ y_2y_3 \neq 0\\ H_K(y_0y_3, y_1y_3, y_1^2 y_2, y_0^2 y_2) \leq B}} 1.
	\end{equation}
	Let $\kappa :=  \max_{j=0,1} \{|y_j|\}$ and write $\alpha := y_3/y_2$. The height of two representatives of the same point is the same so we have
	$$
	H_K(y_0y_3, y_1y_3, y_1^2 y_2, y_0^2 y_2) = H_K(y_0 \alpha, y_1 \alpha, y_1^2,y_0^2) = \mathfrak N(\mathfrak J(\alpha,1)) \prod_{i=1}^2 \left\{\kappa \left| \alpha^{(i)}\right|, \kappa^2 \right\}.
	$$
	Following the proof of Theorem 3 of \cite{Schmidt} in section 9, we obtain
	$$
	\sum_{\substack{(y_0\colon y_1) \in \Ps^1 (\Q)\\y_0y_1 \neq 0\\ \kappa\leq B^{1/4}}}\left( \frac{2c_{\Sym^2\Ps^1, \Q}}{ \kappa^{9}}B^3 + O\left( \frac{B^2}{\kappa^6}\right)\right),
	$$
	where $c_{\Sym^2\Ps^1,\Q} = 4\zeta(3)^{-1}$.
	Extending the sum to infinity yields no dominant contribution to the error term. We obtain $Z_{\Q}(\Ps^1,9)$ from the sum over $\kappa$ but must deduct 2 from it since we are missing two points from $\Ps^1(\Q)$. We account for the three missing lines with vanishing coordinates and obtain the desired main term and therefore Lemma \ref{lem:sym2} (a).
	
	Next, we treat the second sum. We proceed as in Section \ref{sec:Para} and take the sum of \eqref{eq:EndOfSectionSum} over all quadratic number fields. The only thing that is different from the previous treatment is that we do not get the whole height zeta function of $\Ps^1(K)$, instead we get
	$$
	N_{U_2}(B) = \sum_{[K \colon \Q] =2} \sum_{\substack{\mathbf x \in \Ps^1(K) \setminus \Ps^1(\Q)}} \frac{c_{\Ps^1,K}}{H_K(\mathbf x)^3} B^2 + O\left(\frac{c_{\Ps^1, K}}{\zeta_K\left(3/2\right)}B^{3/2}\log B\right).
	$$
	We denote the sum over $\frac{1}{H_K(\mathbf x)^3}$ in the main term by the height zeta function $Z_{K\setminus \Q}(\Ps^1,3)$.
	We continue as in Section 8 of \cite{Schmidt} and order the quadratic number fields $K$ by their discriminant $\Delta_K$. We impose a bound an the absolute value of $\Delta$ by using \cite[Theorem 2]{Silverman}. Let $\mathbf x \in \Ps^3(\bar K)$ with $[\Q(\mathbf x) \colon \Q] =2$, then the discriminant $\Delta$ of $\Q(\mathbf x)$ has the property
	$$
	|\Delta| \leq 4 H_K(\mathbf x)^2.
	$$
	Let $\mathcal D^\pm$ be the set of fundamental discriminants with $\Delta >1$ or $\Delta <0$, respectively. We denote elements of $\mathcal D^\pm$ that are bounded in absolute value by $Y$ by $\mathcal D^\pm(Y)$. We change the index of the field invariants to $\Delta$ and denote by $K_{\Delta}$ the quadratic number field with discriminant $\Delta$. We split the sum and obtain
	\begin{equation}\label{eq:AusgangsDelta}
	\begin{split}
	N_{U_2}(B) = \sum_{ \Delta \in \mathcal D^\pm( 4B^2)}\left(c_{\Ps^1, K_\Delta} Z_{K_\Delta\setminus \Q}\left(\Ps^1,3\right) B^2 + O\left(\frac{c_{\Ps^1,K_{\Delta}}}{\zeta_{\Delta}(3/2)}B^{3/2}\log B\right)\right) .
	\end{split}
	\end{equation}
	We are left with the complete height zeta function in the main term. We want to obtain an asymptotic formula that is dependent on the field discriminant. Therefore we use Abel's summation formula and obtain
	$$
	\sum_{\substack{\mathbf x \in \Ps^1(K_{\Delta})\\ \deg \mathbf x = 2}} \frac1{H_{\Delta}(\mathbf x)^3} = \lim_{B\rightarrow \infty} \frac{N_{\Ps^1, K_\Delta\setminus \Q}(B)}{B^3} + 3  \int_1^B \frac{N_{\Ps^1, K_\Delta\setminus \Q}(u)}{u^4} du,
	$$
	We use \cite[Theorem 2]{Schmidt} which yields $N_{\Ps^1, K_{\Delta}\setminus \Q}(B) = c_{\Ps^1, K_{\Delta}} B^2 + O( \hat{c}_{\Ps^1,K_{\Delta}} B^{3/2})$, where 
	$$
	\hat{c}_{\Ps^1,K_{\Delta}} = \frac1{\sqrt{|\Delta|}} \left( h_{\Delta} R_{\Delta}\max\{ \log\left(h_{\Delta}R_{\Delta}\right),1\} \right)^{1/2}.
	$$
	With this we obtain for the height zeta function
	$$
	\sum_{\substack{\mathbf x \in \Ps^1(K_{\Delta})\\ \deg \mathbf x \leq 2}} \frac1{H_{\Delta}(\mathbf x)^3} = 3 c_{\Ps^1,K} + O\left(\hat{c}_{\Ps^1,K_{\Delta}}\right).
	$$
	Hence,  we can rewrite \eqref{eq:AusgangsDelta} to
	\begin{equation}\label{eq:NSYM2}
	N_{U_2}(B) \ll 
	\sum_{\Delta\in \mathcal D^{\pm}( 4B^2)} \left(c_{\Ps^1, K_\Delta}^2+ c_{\Ps^1, K_\Delta}\hat{c}_{\Ps^1,K_{\Delta}}\right)B^2 +  \frac{ c_{\Ps^1, K_\Delta}}{\zeta_{\Delta}(3/2)} B^{3/2} \log B
	\end{equation}
	Clearly, $\hat{c}_{\Ps^1, K_\Delta} \ll c_{\Ps^1, K_\Delta}$ as $\Delta$ tends to $\infty$, so the dominant contribution is coming from $c_{\Ps^1, K_\Delta}^2$. 
	We can relate $c_{\Ps^1, K_\Delta}$ to a quotient of $L$-functions and use the works by Schmidt \cite{Schmidt} and Le Rudulier \cite{RudulierThesis} to bound these.
	
	Let $(\frac{\Delta}{\ell})$ be the Kronecker symbol. It makes sense to study the $L$-function belonging to the quadratic field with discriminant $\Delta$ which is the $L$-function attached to the Kronecker symbol. It is defined by
	$$
	L(s,\Delta) := \sum_{\ell = 1}^{\infty} \left( \frac{\Delta}{\ell}\right) \ell^{-s}.
	$$
	To connect $c_{\Ps^1, K_\Delta}$ to $L(s,\Delta)$, we use the two classical identities
	$$
	\zeta_{\Delta}(s) = \zeta(s)L(s,\Delta),\hspace{10mm}
	\frac{2^r(2\pi)^s R_{\Delta} h_{\Delta}}{\omega_{\Delta}\sqrt{|\Delta|}} = L(1,\Delta),
	$$
	\cite[(137)]{Hecke} and \cite[(145)]{Hecke}, respectively.
	Combining both equalities yields
	$$
	\frac{c_{\Ps^1, K_\Delta}}{2^{r+s-1}}=\frac{2^{2r}(2\pi)^{2s} h_{\Delta}R_{\Delta}}{|\Delta| \omega_{\Delta} \zeta_{\Delta}(2)} = \frac{2^r (2\pi)^s L(1,\Delta)}{\sqrt{|\Delta|} \zeta(2) L(2,\Delta)}.
	$$ 
	Besides, we get $\zeta_{\Delta}(3/2)^{-1} = \zeta(3/2)^{-1} L(3/2,\Delta)^{-1}$. Using Möbius inversion we obtain
	$$
	L\left(\frac32, \Delta\right) \sum_{m=1}^\infty \mu(m) \left(\frac{\Delta}{m}\right) m^{-3/2} = \sum_{c=1}^\infty \left(\frac{\Delta}{n}\right) n^{-3/2}\sum_{a|c} \mu(a) = 1.
	$$
	This implies
	$$
	L\left(\frac32, \Delta\right)^{-1} = \sum_{m=1}^\infty \mu(m) \left(\frac{\Delta}{m}\right) m^{-3/2} \leq \zeta\left(\frac32\right). 
	$$
	So both summands of \eqref{eq:NSYM2} solely depend on  $c_{\Ps^1,K}$ and $c_{\Ps^1,K}^2$. We obtain trivial bounds for $2^r(2\pi)^s2^{r+s-1}$ since these invariants are only depend on the sign of the discriminant. So we need estimates for
	\begin{align*}
	\sum_{\Delta \in \mathcal D^{\pm}(Y)} \frac{L(1,\Delta)}{\sqrt{|\Delta|}L(2,\Delta)} \text{ and } \sum_{\Delta \in \mathcal D^{\pm}(Y)} \frac{L(1,\Delta)^2}{|\Delta|L(2,\Delta)^2}.
	\end{align*}
		By the proposition of the appendix of \cite{Schmidt} we know that
		$$
		\sum_{ \Delta \in \mathcal D^\pm(Y)} \frac{L(1,\Delta)}{L(2,\Delta)} \ll Y.
		$$
		We use Abel's summation formula together with Schmidt's proposition to obtain
		\begin{align}\label{eq:SchmidtL}
		\sum_{ \Delta \in \mathcal D^\pm(Y)} \frac{L(1,\Delta)}{\sqrt{|\Delta|}L(2,\Delta)} \ll \sqrt{Y} + \int_1^Y \frac{u}{2u^{3/2}} du = 2\sqrt{Y}.
		\end{align}
		In the same way, we use \cite[Introduction, Théorème 5]{RudulierThesis} which yields the bound
		$$
		\sum_{ \Delta \in \mathcal D^\pm(Y)} \frac{L(1,\Delta)^2}{L(2,\Delta)^2} \ll Y.
		$$
		We use Abel's summation formula again and obtain
		\begin{equation}\label{eq:RedulierL}
		\sum_{ \Delta \in \mathcal D^\pm(Y)} \frac{L(1,\Delta)^2}{|\Delta| L(2,\Delta)^2} \ll 1 + \int_1^Y \frac1u du = \log Y + O(1).
		\end{equation}
	We combine \eqref{eq:NSYM2} with \eqref{eq:SchmidtL} and \eqref{eq:RedulierL} to obtain
	\begin{equation}\label{eq:ratpoint22}
	N_{U_2} (B) \ll B^2 \log B.
	\end{equation}
	We combine the count on $U_2$ together with the previously established count of rational points of type 1 and obtain Lemma \ref{lem:sym2} (b).
	
\end{proof}

\section{Interpreting the point counts on other symmetric squares}

Le Rudulier \cite{RudulierThesis} did the point count for the symmetric square of other rational surfaces. However she did not interpret her results in terms of the conjecture by Batyrev and Tschinkel. Let us first state her result. Note that our height comes from an adelic metric on $\mathcal O(1)$ while Le Rudulier uses the height on $\omega_{\Ps^1}^\vee$. This explains the difference between our version and the original result.

\begin{Theorem}
Consider the usual adelic measure on $\mathcal O(1)$ on $\Ps^1$ over $\Q$. This induces an adelic measure on $\Sym^2 \left( \Ps^1 \times \Ps^1\right)$. The following result
\[
N_{\Sym^2 \left( \Ps^1 \times \Ps^1\right), \Q}(B) \sim 2 c_{\Sym^2 \Ps^1,\Q} Z(\Ps^1,6) B^3
\]
from \cite{RudulierThesis} agrees with the prediction by Batyrev and Tschinkel.
\end{Theorem}

One can show that $\Sym^2(\Ps^1 \times \Ps^1)$ is a Fano fourfold and Le Rudulier did the count and checked that the removal of a certain thin subset $\mathcal Z \subseteq \Sym^2(\Ps^1 \times \Ps^1)(\Q)$ lowers the order of the point count to $B^2 \log B$ as predicted by Manin's conjecture. We will show that the order $B^3$ of the total point count is explained by the conjecture by Batyrev and Tschinkel. In this case the many lines $\Lambda$ on $\Ps^1 \times \Ps^1$ give projective planes $\Sym^2 \Lambda$ on $\Sym^2(\Ps^1 \times \Ps^1)$ which explain the count.

\begin{proof}
The proof of Th\'eor\`eme~4.6 in \cite{RudulierThesis} shows that main contribution for the point count comes from points $(\mathbf x,\bar{\mathbf x})$ for an $\mathbf x=(\mathbf x_0,\mathbf x_1) \in (\Ps^1 \times \Ps^1)(\bar{\Q})$ for which one of the points $\mathbf x_0$ and $\mathbf x_1$ is rational point on the projective line the other a quadratic point. This actually gives rise to two naturally weakly saturated subsets depending on which coordinate is rational. Both subscheme have a natural filtration given by the rational coordinate. We will show that the fibres of these filtrations are strongly saturated varieties isomorphic to $\Ps^2$. 

Let $\Lambda_{\mathbf y}$ be a fibre of either projection $\Ps^1 \times \Ps^1 \to \Ps^1$ of a point $\mathbf y \in \Ps^1(\Q)$. We consider the Batyrev--Tschinkel conjecture for $\Sym^2 \Lambda_{\mathbf y} \subseteq \Sym^2(\Ps^1 \times \Ps^1)$ for the induced adelic metric. As in the proof of Proposition~\ref{prop:sym2VinBTterms} we see that $\mathcal L^{(2)}$ on $\Sym^2(\Ps^1 \times \Ps^1)$ restricts to $\mathcal O(1)$ on $\Lambda_{\mathbf y, \mathbf y} := \Sym^2 \Lambda_{\mathbf y} \cong \Ps^2$. As before this shows that the expected order of growth is $B^3$.

For the constant we only need to compare the local Tamagawa numbers to those of $\Sym^2 \Ps^1$ with the adelic metric induced from the usual one on $\Ps^1$. We again fix a representation $\mathbf y=(y_0,y_1)$ and let $y_{\textup m, \nu}$ be one of the $y_i$ for which $|y_i|_\nu$ is maximal. 

\begin{align*}
\tau_\nu(\Sym^2 \Lambda_{\mathbf y}) & = \int_{\Lambda_{\mathbf y, \mathbf y}} \omega_{\Lambda_{\mathbf y, \mathbf y},\nu}\\
	& = \int_{\mathbb A^2(K_\nu)} \prod_{\substack{\text{roots $\tau$ of}\\ T^2-t_1T+t_2}} \frac{dt_1\ dt_2}{\sup\left\{\left|y_0\right|_\nu, \left|y_1\right|_\nu\right\}^{3[K_\nu \colon \Q_p]}\sup\left\{\left|\tau\right|_\nu, 1 \right\}^{3[K_\nu \colon \Q_p]}}\\
& = |y_{\textup{m},\nu}|_\nu^{-6[K_\nu \colon \Q_p]} \int_{\mathbb A^2(K_\nu)} \prod_{\substack{\text{roots $\tau$ of}\\ T^2-t_1T+t_2}} \frac{dt_1\ dt_2}{\sup\left\{\left|\tau\right|_\nu, 1 \right\}^{3[K_\nu \colon \Q_p]}}\\
& = |y_{\textup{m},\nu}|_\nu^{-6[K_\nu \colon \Q_p]} \tau_\nu(\Sym^2 \Ps^1_{\Q}).
\end{align*}
We conclude that $\tau(\Sym^2 \Lambda_{\mathbf y})=\frac{\tau(\Sym^2 \Ps^1_{\Q})}{H(\mathbf y)^6}$ and hence
\[
c(\Sym^2 \Lambda_{\mathbf y})=\frac{c(\Sym^2 \Ps^1_{\Q})}{H(\mathbf y)^6}.
\]
Summing over all $\mathbf y \in \Ps^1(\Q)$ and both projections we see that the result agrees with the count done by Le Rudulier.
\end{proof}

In contrast, the same is not true for counting rational points on $\Sym^2 \Ps^2$. Th\'eor\`eme~4.2 in \cite{RudulierThesis} shows that the Manin--Peyre conjecture is only true after removing a thin set of type II. Hence the Batyrev--Tschinkel conjecture cannot be satisfied; in the current formulation it naturally deals with some type I thin sets as for $\Sym^2(\Ps^1 \times \Ps^1)$. However, any Fano variety which only satisfies the Manin--Peyre conjecture after removing a thin set of type II will not satisfy the Batyrev--Tschinkel conjecture. A refined version which also deals with thin sets of type II is still missing.

\end{document}